\documentclass[11pt,leqno,oneside,letterpaper, openany]{amsart}
\usepackage[letterpaper,left=1.75truein, top=1truein]{geometry}
\usepackage{amsmath,amsfonts,amssymb,amscd,amsthm,amsbsy,epsf,graphics} 
\usepackage{stmaryrd}
\usepackage{setspace}
\usepackage[title]{appendix}
\usepackage{comment} 
\usepackage[normalem]{ulem}

\usepackage[dvipsnames]{xcolor}
\usepackage{color}

\definecolor{org}{RGB}{225,119,0}

\usepackage[colorlinks,linkcolor=blue,citecolor=blue]{hyperref} 

\usepackage{bm}
\usepackage{caption,subcaption,graphicx,epsfig}
\usepackage{tikz-cd,calc}
\usetikzlibrary{knots}
\usetikzlibrary{decorations.markings}

\usepackage{lipsum}

\makeatletter
\def\l@subsection{\@tocline{2}{0pt}{4pc}{5pc}{}}
\makeatother
 
\let\oldtocsection=\tocsection
 
\let\oldtocsubsection=\tocsubsection
 
\let\oldtocsubsubsection=\tocsubsubsection
 
\renewcommand{\tocsection}[2]{\hspace{0em}\oldtocsection{#1}{#2}}
\renewcommand{\tocsubsection}[2]{\hspace{0em}\oldtocsubsection{#1}{#2}}
\renewcommand{\tocsubsubsection}[2]{\hspace{2em}\oldtocsubsubsection{#1}{#2}}

\DeclareGraphicsExtensions{.pdf}

\newtheorem{theorem}{Theorem}[section] 
\newtheorem*{theorem*}{Theorem}

\newtheorem*{corollary*}{Corollary}
\newtheorem{lemma}[theorem]{Lemma}

\newtheorem*{proposition*}{Proposition}

\newtheorem{question}[theorem]{Question}
\newtheorem*{question*}{Question}

\theoremstyle{definition}

\newtheorem*{problem*}{Problem}

\newtheorem*{remark*}{Remark}

\newtheorem*{acknowledgement*}{Acknowledgements}

\newtheoremstyle{cases}
  {12pt plus 6 pt}%       Space above
  {2pt}%       Space below
  {\bfseries}   %       Body font
  {}%          Indent amount (empty = no indent, \parindent = para indent)
  {\bfseries}% Thm head font
  {.}%         Punctuation after thm head
  {.5em}%      Space after thm head: " " = normal interword space;
  {}%          Thm head spec (can be left empty, meaning `normal')
\theoremstyle{cases}

\numberwithin{subcase}{case} 
\numberwithin{subsubcase}{subcase}
\numberwithin{equation}{subsection} 

\graphicspath{{figures/}} 
\begin{document}

\title{On the homeomorphism problem for 4-manifolds}

\author[Cameron McA. Gordon]{Cameron McA. Gordon} 
\address{Department of Mathematics, University of Texas at Austin, 1 University Station, Austin, TX 78712, USA.}
\email{gordon@math.utexas.edu}

%\thanks{2010 Mathematics Subject Classification.  Primary 57M25, 57M50, 57M99}

%\thanks{Key words: }

\begin{abstract}
We show that there is no algorithm to decide whether or not a given 4-manifold is homeomorphic to the connected sum of 12 copies of $S^2 \times S^2$.
\end{abstract} 

\maketitle

\section{Introduction}
In [Mar] Markov showed that the homeomorphism problem for closed 4-manifolds is algorithmically unsolvable. In fact he showed that for some integer $k$ the {\it recognition problem} for $\#_k (S^2 \times S^2)$, the connected sum of $k$ copies of $S^2 \times S^2$, is unsolvable, i.e. there is no algorithm to decide whether or not a given 4-manifold is homeomorphic to $\#_k (S^2 \times S^2)$. 

To describe this in more detail, let us define a {\it $k$-relator Adjan-Rabin set} to be a recursively enumerable set $\mathcal P$ of finite $k$-relator group presentations, such that there is no algorithm to decide whether or not the group presented by a given $P \in \mathcal P$ is trivial. Such sets were shown to exist, for some $k$, by Adjan [A] and Rabin [Ra], using the existence, proved by Novikov [N] and Boone [Boo], of a finitely presented group with unsolvable word problem. Markov showed that if there exists a $k$-relator Adjan-Rabin set then the recognition problem for $\#_k S^2 \times S^2$ is unsolvable.

The author showed [G] that from a finite $m$-relator presentation of a group with unsolvable word problem one can construct an $(m + 2)$-relator Adjan-Rabin set; there is also an account of this work in the survey article [Mi]. In [Bor] Borisov constructed a finite 12-relator presentation of a group with unsolvable word problem. It follows that the recognition problem for $\#_{14} (S^2 \times S^2)$ is unsolvable. See [S1],[S2],[CL].

The purpose of the present note is to offer the following improvement.     

\begin{theorem}
\label{thm: rec prob}
The recognition problem for $\#_{12} (S^2 \times S^2)$ is unsolvable.
\end{theorem} 

A natural question is whether $k$ can be reduced further, in particular whether it can be reduced to 0.

\begin{question}
Is the recognition problem for $S^4$ unsolvable?
\end{question}
The recognition problem for $S^n$ is unsolvable for $n \ge 5$ [VKF, Appendix by S.P. Novikov], and solvable for $n \le 3$ [Ru],[T]. 

The proof of Theorem 1.1 has two parts, one algebraic and the other topological, each enabling $k$ to be reduced by 1. The first is discussed in Section 2, and the second in Section 3.

\section{The Algebra} 

Let $(x_1,...,x_n : r_1,...,r_m)$ be a finite presentation of a group $G$. Let $\bar x_i$ denote the image of $x_i$ in $G/[G,G]$.

Consider the following property:

(2.1) \; there exists $p$, $1 \le p \le n$, such that for $1 \le i \le p$, $\bar x_i$ has finite order $q_i \ge 1$, where gcd$(q_1,...,q_p) = 1$. 

\begin{lemma}
If there exists a group with unsolvable word problem having a finite $m$-relator presentation that satisfies (2.1), then there exists an $(m + 1)$-relator Adjan-Rabin set.
\end{lemma}
\begin{proof}
We modify the construction given in [G]. Let $(x_1,...,x_n : r_1,...,r_m)$ be a presentation of a group $G$ with unsolvable word problem that satisfies (2.1). By taking a minimal set $\{x_1,...,x_p\}$ with property (2.1) we may assume that the $q_i$ are all distinct. Let $q =$ max$\{q_1,...,q_p\}$.

Let $W(x_1,...,x_n)$ denote the set of words in $\{x_1,...,x_n\}$, i.e. the set of expressions of the form $x^{\epsilon_1}_{i_1}...x^{\epsilon_r}_{i_r}$, $x_{i_j} \in \{x_1,...,x_n\}$, $\epsilon_j = \pm 1$. For $w \in W(x_1,...,x_n\}$, let $Q_w$ be the presentation with generators $x_1,...,x_n, a, \alpha, b, \beta$, and relators $r_1,...,r_m$ together with 

(i)    $a \alpha a^{-1} = b^2$

(ii)   $\alpha a \alpha^{-1} = b \beta b^{-1}$

(iii)  $a^{-q_i} x_i \alpha^{q_i} = \beta^{-i} b \beta^i$, $\quad$  $1 \le i \le p$

(iv)   $a^{-(q+i)} x_i \alpha^{(q+i)} = \beta^{-i} b \beta^i$, $\quad$ $p+1 \le i \le n$

(v)    $[w,\alpha^2] = \beta^{-(n+1)} b \beta^{(n+1)}$

where $[x,y]$ means $xyx^{-1}y^{-1}$.

Let $G_w$ be the group presented by $Q_w$. We can apply the following Tietze transformations to $Q_w$. Using (i), express $\alpha$ in terms of $a$ and $b$, substitute this expression for the occurrences of $\alpha$ in the other relations, then delete $\alpha$ from the generators and (i) from the relations. Now from (ii) express $\beta$ in terms of $a$ and $b$, substitute for $\beta$ into the other relations, and delete $\beta$ and relation (ii). Using relations (iii) and (iv) we can now write the $x_i$ as words in $a$ and $b$, substitute these into the relators $r_j$, getting relators $r'_j$ that are words in $a$ and $b$, substitute for the $x_i$ in $w$ in (v), and finally delete the $x_i$ and relations (iii) and (iv).

We are left with a presentation $P_w$ of $G_w$ with two generators, $a$ and $b$, and $(m+1)$ relations: the relators $r'_j$, $1 \le j \le m$, and the transformed relation (v). We claim that $\{P_w : w \in W(x_1,...,x_n)\}$ is an Adjan-Rabin set.

Let $U$ denote the set of elements listed on the right-hand side of the relations (i) - (v). By examining the possible cancellation in $u_1^{\epsilon_1}u_2^{\epsilon_2}$, where $u_1$ and $u_2$ are distinct elements of $U$ and $\epsilon_i = \pm1$, $i = 1,2$, it is easy to see that a non-empty reduced word in the elements of $U$ has positive length when expressed as a reduced word in $b$ and $\beta$. Thus $U$ is a basis for a free subgroup of the free group $F(b, \beta)$. Similarly, if $[w] \ne 1$ in $G$, one sees that the set of elements on the left-hand side of the relations is a basis for a free subgroup of the free product $G \ast F(a, \alpha)$. Hence if $[w] \ne 1$ in $G$ then $G_w$ is a free product with amalgamation $(G \ast F(a, \alpha)) \ast_{F} F(b, \beta)$, where $F$ is free of rank $(n+3)$. In particular $G_w \ne 1$.

If $[w] = 1$ in $G$ then (v), together with the relators $r'_1,...,r'_m$, implies that $b = 1$, and therefore $G_w$ is cyclic, generated by $a$. Also, $\alpha = 1$ by (i). 
    
Relations (iii) give $x_i = a^{q_i}$, $1 \le i \le p$. By condition (2.1) $x_i$ maps to an element in $G_w$ of order dividing $q_i$; hence in $G_w$ we have the relations $a^{q^{2}_i}$, $1 \le i \le p$. Since gcd$(q_1,...,q_p) = 1$, this implies $a = 1$, and hence $G_w = 1$. 

Thus $G_w = 1$ if and only if $[w] = 1$ in $G$. Since $G$ has unsolvable word problem, $\{P_w : w \in W(x_1,...,x_n)\}$ is an Adjan-Rabin set.                                       
\end{proof}

\begin{theorem}
There exists a 13-relator Adjan-Rabin set.
\end{theorem}
\begin{proof}
Matijasevi\v c [Mat] has shown that there exists a semigroup $S$ having a presentation with two generators and three relations, and a positive word $W_0$ in the generators, such that there is no algorithm to decide, for an arbitrary positive word $W$ in the generators, whether or not $W$ and $W_0$ represent the same element of $S$. Borisov shows that this may be used to construct a presentation, with generators $a,b,c,d,$ and $e$ and 12 relations, of a group $\Gamma'$ with unsolvable word problem; see [Bor, $\S 3$]. Among the relations are 
$$\mu_{i} d \mu^{-1}_i = d^{\alpha}, \;  \mu^{-1}_{i} e \mu_i = e^{\alpha}, \quad i = 1,2$$   

where $\mu_1$ and $\mu_2$ are words in $a$ and $b$ and $\alpha$ is an arbitrary integer $> 3$. However, an examination of the proof in [Bor] that $\Gamma'$ has unsolvable word problem shows that these relations may be replaced by
$$\mu_{i} d \mu^{-1}_i = d^u, \; \mu^{-1}_i e \mu_i = e^v, \quad i = 1,2$$

for any integers $u,v > 3$.

So, taking $u = 4$, $v = 5$, we get a 12-relator presentation of a group with unsolvable word problem where the generators $d$ and $e$ have the property that the order of $\bar d$ divides 3 and the order of $\bar e$ divides 4. The result now follows from Lemma 2.1.
\end{proof}

\section{The Topology}

We briefly summarize Markov's argument [Mar]. For other discussions see [S1],[CL],[K]. We will not discuss the algorithmic aspects of the PL constructions involved; these are dealt with in [BHP]; see also [S1].

Let $P = (x_1,...,x_n : r_i,...,r_k)$ be a finite presentation of a group $G_P$.

Attach $n$ 1-handles to $B^5$ so as to get an orientable 5-manifold $V$ with $\pi_1(V) \cong F(x_i,...,x_n)$. Let $\gamma_1,...,\gamma_k$ be disjoint circles in $\partial V$ such that $[\gamma_j]$ is conjugate to $r_j$ in $\pi_1(V) \cong \pi_1(\partial V)$, $1 \le j \le k$. Since homotopy implies isotopy for 1-manifolds in a 4-manifold by general position, $\boldsymbol{\gamma} = \bigcup_{j=1}^k \gamma_j$ is well-defined up to isotopy in $\partial V$. Let $N_P$ be obtained by attaching 2-handles $H(\gamma_j)$ to $V$ along $\gamma_j$, $1 \le j \le k$. We express this as $N_P = V \cup H(\boldsymbol{\gamma})$. Clearly $\pi_1(N_P) \cong G_P$. Also, since $N_P$ has a 2-dimensional spine, a general position argument shows that inclusion $\partial N_P \to N_P$ induces an isomorphism on fundamental groups.

The homeomorphism type of $N_P$ depends only on $P$ and a choice of framing ($\in \mathbb{Z}_2$) of the normal bundle of $\gamma_j$ in $\partial V$, $1 \le j \le k$. To ensure that it depends only on $P$ we note that there is an obvious embedding of $V$ in $\mathbb{R}^5$. Then $\gamma_j$ bounds a disk in $\mathbb{R}^5$ and we choose the framimg to be the 0-framing, i.e. the one that extends over the normal bundle of the disk.

Let $\alpha_1,...,\alpha_n$ be disjoint circles in $\partial N_P$ that bound disjoint disks in $\partial N_P$, and let $W_P$ be the result of attaching 2-handles $H(\alpha_i)$ to $N_P$ along $\alpha_i$, $1 \le i \le n$, using the 0-framing. Let $M_P = \partial W_P$. Then $\pi_1(M_P) \cong \pi_1(W_P) \cong \pi_1(N_P) \cong G_P$.

Markov's key observation is the following.  

\begin{lemma}{\rm(Markov)}
$G_P = 1$ if and only if $M_P \cong \#_k (S^2 \times S^2)$.
\end{lemma}

\begin{proof}
Since $\pi_1(M_P) \cong G_P$ the ``if''direction is clear.

For the converse, suppose $G_P = 1$. Let $\beta_1,...,\beta_n$ be disjoint circles in $\partial V$ that are dual to the co-cores of the 1-handles and disjoint from $\boldsymbol{\gamma}$. Then we may regard $\beta_1,...,\beta_n$ as lying in $\partial N_P$. Recalling that $\pi_1(\partial N_P) \cong G_P = 1$, $\boldsymbol{\alpha} = \bigcup_{i = 1}^n \alpha_i$ is isotopic to $\boldsymbol{\beta} = \bigcup_{i=1}^n \beta_i$ in $\partial N_P$. Therefore
\begin{align*}
W_P &= N_P \cup H(\boldsymbol{\alpha})\\
&\cong N_P \cup H(\boldsymbol{\beta})\\
&=(V \cup H(\boldsymbol{\gamma})) \cup H(\boldsymbol{\beta})\\
&=(V \cup H(\boldsymbol{\beta})) \cup H(\boldsymbol{\gamma})\\
&\cong B^5 \cup H(\boldsymbol{\gamma})\\
&\cong \natural_k (S^2 \times D^3)
\end{align*}

where $\natural$ denotes boundary connected sum.

Hence $M_P = \partial W_P \cong \#_k (S^2 \times S^2)$.
\end{proof}

The above construction gives an algorithm that takes a finite $k$-relator presentation $P$ of a group $G_P$ and produces a closed 4-manifold $M_P$ such that $G_P = 1$ if and only if $M_P \cong \#_k (S^2 \times S^2)$. To complete the proof of Markov's theorem we note that if $\mathcal {P}$ is a $k$-relator Adjan-Rabin set then an algorithm to decide, for a given $P \in \mathcal{P}$, whether or not the manifold $M_P$ is homeomorphic to $\#_k (S^2 \times S^2)$ would give an algorithm to decide whether or not $G_P = 1$, a contradiction.

We now describe a modification of the proof of Lemma 3.1 that enables us to replace $\#_k (S^2 \times S^2)$ by $\#_{(k-1)} (S^2 \times S^2)$.

Let $\boldsymbol{\alpha'} = \bigcup_{i=1}^{n-1} \alpha_i$, $\boldsymbol{\beta'} = \bigcup_{i=1}^{n-1} \beta_i$, define $W'_P = N_P \cup H(\boldsymbol{\alpha'})$, and let $M'_P = \partial W'_P$. Note that $\pi_1(M'_P) \cong \pi_1(W'_P) \cong \pi_1(N_P) \cong G_P$.

\begin{lemma}
$G_P = 1$ if and only if $M'_P \cong \#_{(k-1)} (S^2 \times S^2)$.
\end{lemma}
\begin{proof} 
As before, the ``if'' direction is clear.

Assume $G_P = 1$. Then, since $\pi_1(\partial N_P) = 1$ $\boldsymbol{\alpha'}$ is isotopic to $\boldsymbol{\beta'}$ in $\partial N_P$, and as in the proof of Lemma 3.1
$$ W'_P \cong (V \cup H(\boldsymbol{\beta'})) \cup H(\boldsymbol{\gamma})$$
which is homeomorphic to $(S^1 \times D^4) \cup H(\boldsymbol{\gamma})$.

Let $a_j = [\gamma_j] \in \pi_1(S^1 \times D^4) \cong \mathbb{Z}$, $1 \le j \le k$. Orient $\gamma_j$ so that $a_j \ge 0$. Since $\pi_1(W'_P) = 1$, gcd$(a_1,...,a_k) = 1$. Therefore, by a sequence of moves of the form 
\begin{align*}
a_r & \mapsto a_r - a_s\\
a_j & \mapsto a_j, \; j \ne r,
\end{align*}
for some $r$ and some $s \ne r$ with $a_s \le a_r$, followed by a permutation, we can transform $(a_1,...,a_k)$ to $(1,0,...,0)$.

Since the above move can be realized by sliding $H(\gamma_r)$ over $H(\gamma_s)$,
$$W'_P \cong (S^1 \times D^4) \cup H(\boldsymbol{\gamma'})$$
where $([\gamma'_1],...,[\gamma'_k]) = (1,0,...,0)$. Thus $W'_P$ is homeomorphic to $B^5$ with $(k-1)$ 2-handles attached with the 0-framing, i.e. $\natural_{(k-1)} (S^2 \times D^3)$. Hence $M'_P \cong \#_{(k-1)} (S^2 \times S^2)$.
\end{proof}

\begin{proof} [Proof of Theorem \ref{thm: rec prob}]
This follows from Theorem 2.2 and Lemma 3.2.
\end{proof}

\end{document}